\documentclass[a4paper,10pt
]{amsart}
\usepackage{amsmath,amssymb,amsthm,}
\usepackage[alphabetic, abbrev]{amsrefs}
\usepackage{enumerate}
\usepackage{hyperref}
\usepackage[all]{xy}
\newdir{ >}{{}*!/-5pt/@{>}} 
\SelectTips{cm}{} 

\usepackage[pagewise,switch]{lineno}

\numberwithin{equation}{section}
\newtheorem{theorem}[subsection]{Theorem}
\newtheorem{corollary}[subsection]{Corollary}
\newtheorem{lemma}[subsection]{Lemma}
\newtheorem{proposition}[subsection]{Proposition}
\theoremstyle{definition}
\newtheorem{definition}[subsection]{Definition}
\newtheorem{remark}[subsection]{Remark}

\newcommand{\cC}{\mathcal{C}}
\newcommand{\cD}{\mathcal{D}}
\newcommand{\cE}{\mathcal{E}}
\newcommand{\cI}{\mathcal{I}}

\newcommand{\cM}{\mathcal{M}}
\newcommand{\cO}{\mathcal{O}}
\newcommand{\cP}{\mathcal{P}}
\newcommand{\cS}{\mathcal{S}}
\newcommand{\sset}{\mathrm{sSet}}
\newcommand{\ssetI}{\mathrm{sSet}^{\cI}}
\newcommand{\Ho}{\mathrm{Ho}}

\DeclareMathOperator{\hocolim}{hocolim}
\DeclareMathOperator{\colim}{colim}
\DeclareMathOperator{\const}{const}
\DeclareMathOperator{\Map}{Map}

\newcommand{\ot}{\leftarrow}

\newcommand{\iso}{\cong}

\newcommand{\bld}[1]{{\mathbf{#1}}}

\newcommand{\id}{{\mathrm{id}}}
\newcommand{\Fun}{{\mathrm{Fun}}}
\newcommand{\SMC}{\mathcal{S}{\mathrm{ym}\mathcal{M}\mathrm{on}\mathcal{C}\mathrm{at}_\infty}}
\newcommand{\Cat}{\mathcal{C}\mathrm{at}}

\newcommand{\Fr}{\mathrm{Fr}}

\newcommand{\arxivlink}[1]{\href{http://arxiv.org/abs/#1}{\texttt{arXiv:#1}}}

\usepackage[pdftex]{graphics,color} 
\usepackage{eso-pic}
\usepackage{scrtime}
\usepackage{ifdraft}
\ifdraft{%
\AddToShipoutPicture{\put(30,30){\resizebox{!}{.4cm}%
   {{{\color[gray]{0.8}\texttt{
compiled \the\year-\the\month-\the\day~at \thistime}}}}}}
}{}

\title[Presentably symmetric monoidal
\texorpdfstring{$\infty$}{infinity}-categories]{Presentably symmetric
  monoidal \texorpdfstring{$\infty$}{infinity}-categories are
  represented by symmetric monoidal model categories}

\author{Thomas Nikolaus}
\address{Max Planck Institute for Mathematics,
Vivatsgasse 7,
53111 Bonn,
Germany}
\email{thoni@mpim-bonn.mpg.de}

\author{Steffen Sagave} \address{IMAPP, Radboud University Nijmegen, PO Box 9010, 6500 GL Nijmegen, The Netherlands} \email{s.sagave@math.ru.nl}

\date{\today}

\begin{document}
\begin{abstract}
We prove the theorem stated in the title. More precisely, we show the stronger statement that every symmetric monoidal left adjoint functor between presentably symmetric monoidal \texorpdfstring{$\infty$}{infinity}-categories is represented by a strong symmetric monoidal left Quillen functor between simplicial, combinatorial and left proper symmetric monoidal model categories.
\end{abstract}

\maketitle

\section{Introduction}
The theory of $\infty$-categories has in recent years become a powerful tool for studying questions in homotopy theory and other branches of mathematics. It complements the older theory of Quillen model categories, and in many application the interplay between the two concepts turns out to be crucial.  In an important class of examples, the relation between $\infty$-categories and model categories is by now completely understood, thanks to work of Lurie \cite[Appendix A.3]{Lurie_HTT} and Joyal \cite{Joyal-quasi-categories} based on earlier results by Dugger \cite{Dugger_combinatorial}: On the one hand, every combinatorial simplicial model category~$\cM$ has an underlying $\infty$-category $\cM_\infty$. This $\infty$-category $\cM_\infty$ is \emph{presentable}, i.e., it satisfies the set theoretic smallness condition of being accessible and has all $\infty$-categorical colimits and limits. On the other hand, every presentable $\infty$-category is equivalent to the $\infty$-category associated with a combinatorial simplicial model category~\cite[Proposition A.3.7.6]{Lurie_HTT}. The presentability assumption is essential here since a sub $\infty$-category of a presentable $\infty$-category is in general not presentable, and does not come from a model category.

In many applications one studies model categories $\cM$ equipped with a symmetric monoidal product that is compatible with the model structure. The underlying $\infty$-category $\cM_\infty$ of such a \emph{symmetric monoidal model category} inherits the extra structure of a \emph{symmetric monoidal $\infty$-category}~\cite[Example 4.1.3.6 and Proposition 4.1.3.10]{Lurie_HA}. Since the monoidal product of $\cM$ is a Quillen bifunctor, $\cM_{\infty}$ is an example of a \emph{presentably symmetric monoidal $\infty$-category}, i.e., a symmetric monoidal $\infty$-category $\cC$ which is presentable and whose associated tensor bifunctor $\otimes\colon \cC \times \cC \to \cC$ preserves colimits separately in each variable. In view of the above discussion, it is an obvious question whether every presentably symmetric monoidal $\infty$-category arises from a combinatorial symmetric monoidal model category. This was asked for example by Lurie~\cite[Remark 4.5.4.9]{Lurie_HA}. The main result of the present paper is an affirmative answer to this question: 

\begin{theorem}\label{main_thm} For every presentably symmetric monoidal $\infty$-category $\cC$, there is a simplicial, combinatorial and left proper symmetric monoidal model category $\cM$ whose underlying symmetric monoidal $\infty$-category is equivalent to $\cC$.
\end{theorem}
One can view this as a rectification result: The a priori weaker and more flexible notion of a symmetric monoidal $\infty$-category, which can encompass coherence data on all layers, can be rectified to a symmetric monoidal category where only coherence data up to degree 2 is allowed. An analogous result in the monoidal (but not symmetric monoidal) case is outlined in~\cite[Remark 4.1.4.9]{Lurie_HA}. The symmetric result is significantly more complicated, as it is generally harder to rectify to a commutative structure than to an associative one. As we will see in Section~\ref{subsec:functoriality} below, the theorem can actually be strengthened to a functorial version stating that symmetric monoidal left adjoint functors are represented by strong symmetric monoidal left Quillen functors.

The strategy of proof for Theorem~\ref{main_thm} is as follows. Using
localization techniques, we reduce the statement to the case of
presheaf categories. By a result appearing in work of
Pavlov--Scholbach~\cite{Pavlov-Scholbach-admissibility}, we can
represent a symmetric monoidal $\infty$-category by an $E_{\infty}$
algebra $M$ in simplicial sets with the Joyal model structure. The
main result of Kodjabachev--Sagave~\cite{Kodjabachev-S_Joyal-I} implies that this
$E_{\infty}$ algebra can be rigidified to a strictly commutative
monoid in the category of diagrams of simplicial sets indexed by
finite sets and injections. We construct a chain of Quillen
equivalences relating the contravariant model structure on $\sset/M$
with a suitable contravariant model structure on objects over the
commutative rigidification of $M$.  The last step provides a symmetric
monoidal model category, and employing a result by
Gepner--Groth--Nikolaus~\cite{GGN} we show that it models the
symmetric monoidal $\infty$-category of presheaves on $M$.

It is also worth noting that our proof of Theorem~\ref{main_thm} does
in fact provide a symmetric monoidal model category $\cM$ with
favorable properties: operad algebras in $\cM$ inherit a model
structure from $\cM$, and weak equivalences of operads induces Quillen
equivalences between the categories of operad algebras; see Theorem~\ref{thm:favorable-properties-of-M} below. In
particular, there is a model structure on the category of commutative
monoid objects in $\cM$ which is Quillen equivalent to the lifted model
structure on $E_{\infty}$ objects in $\cM$ and moreover models the $\infty$-category of commutative algebras in the $\infty$-category represented by $\cM$.  Hence formally $\cM$ behaves very
much like symmetric spectra with the positive model structure.

\subsection{Applications} Our main result allows to abstractly deduce the existence of symmetric monoidal model categories that represent homotopy theories with only homotopy coherent symmetric monoidal structures. For example, it was unknown for a long time if there is a good point set level model for the smash product on the stable homotopy category.  Since a presentably symmetric monoidal $\infty$-category that models the stable homotopy category can be established without referring to such a point set level model for the smash product, the existence of a model category of spectra with good smash product follows from our result. (Explicit constructions of such model categories of course predate the notion of presentably symmetric monoidal $\infty$-categories.)

But there are also examples where the question about the existence of symmetric monoidal models is open. One such example is the category of topological operads. It admits a tensor product, called the \emph{Boardman--Vogt tensor product}, which controls the interchange of algebraic structures. The known symmetric monoidal point set level models for this tensor product cannot be derived, i.e., they do not give rise to a symmetric monoidal model category. However, for the underlying $\infty$-category of $\infty$-operads a presentably symmetric monoidal product is constructed by Lurie~\cite[Chapter~2.2.5]{Lurie_HA}. In this case, our result allows to abstractly deduce the existence of a symmetric monoidal model category modeling operads with the Boardman--Vogt tensor product.

\subsection{Organization} 
In Section~\ref{sec_two} we show that Theorem~\ref{main_thm} and its
functorial enhancement can be reduced to the case of presheaf
categories.  In Section~\ref{sec:contravariant-I} we develop variants
of the contravariant model structure that are compatible with the
rigidification for $E_{\infty}$ quasi-categories recently developed by
Kodjabachev--Sagave~\cite{Kodjabachev-S_Joyal-I}. In the final
Section~\ref{sec:E-infty} we prove that an instance of the
contravariant model structure provides the desired result about
presheaf categories.


\subsection{Acknowledgments} We would like to thank Gijs Heuts, Dmitri Pavlov and Markus Spitzweck for helpful discussions. Moreover, we would like to thank the referee for useful comments.

\section{Reduction to presheaf categories}\label{sec_two}
In this section we explain how Theorem~\ref{main_thm} follows from a statement about presheaf categories that will be established in Section~\ref{sec:contravariant-I}.

As defined by Lurie in \cite[Definition 2.0.0.7]{Lurie_HA}, a symmetric monoidal $\infty$-category is a cocartesian fibration of simplicial sets $\cC^{\otimes} \to N(\mathcal{F}in_*)$ satisfying a certain condition. We explain in Proposition~\ref{prop:SMC-as-localization} below that a symmetric monoidal $\infty$-category can be represented by an $E_{\infty}$ algebra in simplicial sets with the Joyal model structure. We also note that by \cite[Example 4.1.3.6]{Lurie_HA}, every symmetric monoidal model category gives rise to a symmetric monoidal $\infty$-category, and every symmetric monoidal left Quillen functor induces a left adjoint symmetric monoidal functor between the respective $\infty$-categories.

Recall that an $\infty$-category $\cC$ is called \emph{presentable} if it is $\kappa$-accessible for some regular cardinal $\kappa$ and admits all small colimits. In that case we can write $\cC$ as an accessible localization of the category of presheaves $\cP(\cC^ \kappa)$ on the full subcategory $\cC^\kappa \subset \cC$ of $\kappa$-compact objects.  Here we denote the category of presheaves on an $\infty$-category $\cD$ as $\cP(\cD) = \Fun(\cD^{op}, \cS)$ where $\cS = N(\mathrm{Kan}^\Delta)$ is the $\infty$-category of spaces obtained as the homotopy coherent nerve of the simplicially enriched category of Kan complexes.  Moreover $\cC^\kappa$ is essentially small. Replacing $\cC^\kappa$ by a small $\infty$-category $\cD$ we see that every presentable $\infty$-category is equivalent to an accessible localization of the category of presheaves $\cP(\cD)$ on some small $\infty$-category $\cD$. For a detailed discussion of presentable $\infty$-categories and accessible localizations we refer the reader to~\cite[Chapter 5.5]{Lurie_HTT}.

To study a symmetric monoidal analogue of this statement, we recall the following terminology from the introduction. 
\begin{definition}
A symmetric monoidal $\infty$-category $\cC$ is \emph{presentably symmetric monoidal} if $\cC$ is presentable and the associated tensor bifunctor $\otimes\colon \cC \times \cC \to \cC$ preserves colimits separately in each variable.
\end{definition}

For every symmetric monoidal structure on an
$\infty$-category $\cD$, the $\infty$-category $\cP(\cD)$ inherits a
symmetric monoidal structure which
by~\cite[Corollary~4.8.1.12]{Lurie_HA} is uniquely determined by the
following two properties:
\begin{itemize}
 \item The tensor product makes $\cP(\cD)$ into a presentably symmetric monoidal $\infty$-category.
 \item The Yoneda embedding $j\colon \cD \to \cP(\cD)$ can be extended to a symmetric monoidal functor.
\end{itemize}
We call this structure the \emph{Day convolution symmetric monoidal structure}. It follows from \cite[4.8.1.10(4)]{Lurie_HA} that it has the following universal property: for every presentably symmetric monoidal $\infty$-category $\cE$, the Yoneda embedding ${j\colon \cD \to \cP(\cD)}$ induces an equivalence
\begin{equation*}
\Fun^{L,\otimes}(\cP(\cD), \cE) \to \Fun^\otimes(\cD,\cE).
\end{equation*}
Here $\Fun^\otimes$ denotes the $\infty$-category of symmetric monoidal functors and $\Fun^{L,\otimes}$ denotes the $\infty$-category of functors which are symmetric monoidal and in addition preserve all small colimits 
(or equivalently, which are left adjoint).

In order to state our first structure result for presentably symmetric monoidal $\infty$-categories, let us recall the notion of a symmetric monoidal localization of a symmetric monoidal $\infty$-category $\cC$.  An accessible localization $L\colon \cC \to \cC$ is called \emph{symmetric monoidal} if the full subcategory of local object $\cC^{0} \subseteq \cC$ admits a presentably symmetric monoidal structure such that the induced localization functor $L\colon \cC \to \cC^0$ admits a symmetric monoidal structure. In that case these symmetric monoidal structures are essentially unique. By \cite[Proposition~2.2.1.9]{Lurie_HA}, the localization $L$ is symmetric monoidal precisely if for every local equivalence $X \to Y$ in $\cC$ and every object $Z \in \cC$ the induced morphism $X \otimes Z \to Y \otimes Z$ is also a local equivalence. Note that this condition can be completely checked on the level of homotopy categories. See also \cite[Section 3]{GGN} for a discussion of symmetric monoidal localizations.

\begin{proposition}\label{prop_eins}
Every presentably symmetric monoidal $\infty$-category is an accessible, symmetric monoidal localization of the category of presheaves $\cP(\cD)$ on some small, symmetric monoidal $\infty$-category $\cD$. 
\end{proposition}
\begin{proof}
Let $\cC$ be a presentably symmetric monoidal $\infty$-category. Choose a regular cardinal $\kappa$ such that $\cC$ is $\kappa$ accessible. By enlarging $\kappa$ we can assume that the $\kappa$-compact objects $\cC^\kappa \subset \cC$ form a full symmetric monoidal subcategory. We can replace $\cC^\kappa$ up to equivalence by a small, symmetric monoidal $\infty$-category $\cD$ since it is essentially small. Then we find that $\cC$ is an accessible localization of $\cP(\cD)$. The inclusion $\cD \simeq \cC^\kappa \to \cP(\cD)$ is by construction symmetric monoidal. We conclude that the localization functor $\cP(\cD) \to \cC$ can be endowed with a symmetric monoidal structure with respect to the Day convolution symmetric monoidal structure, using the universal property of the Day convolution. By the description of symmetric monoidal localizations given above this finishes the proof.
\end{proof}

Following~\cite[Definition 1.21]{barwick_left-right} (or rather~\cite[Corollary 2.7]{barwick_left-right}), we say that a combinatorial model category is \emph{tractable} if it admits a set of generating cofibrations with cofibrant domains. 

Now assume that $\cM$ is a simplicial, combinatorial, tractable and left proper symmetric monoidal model category. Denote the underlying symmetric monoidal $\infty$-category by $\cM_\infty$.  Let $L\colon \cM_\infty \to \cM_\infty$ be an accessible and symmetric monoidal localization. We say that a morphism $f\colon A \to B$ in $\cM$ is
\begin{itemize}
 \item a \emph{local cofibration} if it is a cofibration in the original model structure on~$\cM$,
 \item a \emph{local weak equivalence} if $L(\iota f)$ is an equivalence in $\cM_\infty$ where $\iota f$ denotes the corresponding morphism in $\cM_\infty$, and
 \item a \emph{local fibration} if it has the right lifting property with respect to all morphisms in $\cM$ which are simultaneously a cofibration and a weak equivalence. 
\end{itemize}

\begin{proposition}\label{prop_zwei}
  The above choices of local cofibrations, local fibrations and local weak equivalences define a simplicial, combinatorial, tractable and left proper symmetric monoidal model structure.  The underlying $\infty$-category of this model category $\cM^{loc}$ and the $\infty$-category of local objects $L\cM_\infty \subseteq \cM_\infty$ are equivalent as symmetric monoidal $\infty$-categories.
\end{proposition}
\begin{proof}
  We use \cite[Proposition A.3.7.3]{Lurie_HTT} to conclude that $\cM^{loc}$ exists and that it is a simplicial, combinatorial and left proper model category. By construction, it is a left Bousfield localization of $\cM$.  It remains to verify that the local model structure is symmetric monoidal. Since $\cM$ is tractable, so is $\cM^{loc}$, and it follows from~\cite[Corollary 2.8]{barwick_left-right} that we may assume that both the generating cofibrations of $\cM^{loc}$ and the generating acyclic cofibrations of $\cM^{loc}$ have cofibrant domains.  To verify the pushout-product axiom, it therefore suffices to show that on the level of homotopy categories for an object $Z \in \Ho(\cM)$ and a local equivalence $X \to Y$ in $\Ho(\cM)$ the morphism of $X \otimes Z \to Y \otimes Z$ is a local equivalence as well (here the tensor is the tensor on the homotopy category, i.e., the derived tensor product). But this is true since the corresponding fact is true in the $\infty$-category $\cM_\infty$ as discussed above.

By construction the $\infty$-category $L\cM_\infty$ of local objects is modeled by the localized model structure $\cM^{loc}$. It remains to show that the two are equivalent as symmetric monoidal $\infty$-categories. To this end we just observe that the identity is a symmetric monoidal left Quillen functor $\cM \to \cM^{loc}$. Thus the localized model structure endows $L\cM_\infty$ with a symmetric monoidal structure such that the localization $\cM \to \cM^{loc}$ is symmetric monoidal. But this was our defining property of the symmetric monoidal structure on $L\cM_\infty$.
\end{proof}

The next proposition is the technical backbone of this paper and will be proven at the end of Section~\ref{sec:contravariant-I}.
\begin{proposition}\label{prop:monoidal-equivalence}
  Let $\cD$ be a small symmetric monoidal $\infty$-category. Then there exists a simplicial, combinatorial, tractable and left proper symmetric monoidal model category $\cM$ whose underlying presentably symmetric monoidal $\infty$-category is symmetric monoidally equivalent to $\cP(\cD)$ equipped with the Day convolution structure.
\end{proposition}
We can now prove the main theorem from the introduction:
\begin{proof}[Proof of Theorem~\ref{main_thm}]
Propositions~\ref{prop_eins} and~\ref{prop_zwei} reduce the claim to the statement of Proposition~\ref{prop:monoidal-equivalence}.
\end{proof}
The following theorem establishes more properties of the symmetric monoidal model categories that are provided by our proof of Theorem~\ref{main_thm}.
\begin{theorem}\label{thm:favorable-properties-of-M}
Let $\cC$ be a presentably symmetric monoidal $\infty$-category.
Then the symmetric monoidal model category $\cM$ of Theorem~\ref{main_thm} can be chosen such that the following holds:
\begin{enumerate}[(i)]
\item For any operad  $\mathcal O$ in $\sset$, the forgetful functor
$\cM[\mathcal O] \to \cM$ from the category of $\cO$-algebras in $\cM$
creates a model structure on $\cM[\mathcal O]$. 
\item If $\cP \to \cO$ is a weak equivalence of operads,  
then the induced adjunction $\cM[\mathcal P]\rightleftarrows \cM[\mathcal O]$ 
is a Quillen equivalence. In particular, the categories of $E_{\infty}$ objects and strictly commutative monoid objects in $\cM$ are Quillen equivalent. 
\item The $\infty$-category associated with the lifted model structure on commutative monoid objects in $\cM$ is equivalent to the $\infty$-category of commutative algebra objects in the $\infty$-category~$\cC$. 
\end{enumerate}
\end{theorem}
\begin{proof}
Parts (i) and (ii) follow from our construction and Proposition~\ref{prop:operad-lift} below. Part (iii) follows from~\cite[Theorem~7.10]{Pavlov-Scholbach-admissibility}. The \emph{symmetric flatness} hypothesis needed for the latter theorem is verified in the proof of Proposition~\ref{prop:operad-lift} below. 
\end{proof}
\subsection{Functoriality}\label{subsec:functoriality} We now provide a strengthening of our main result for functors. The methods and ideas are precisely the same as before, we only have to carefully keep track of the functoriality.

We first prove a slight generalization of Proposition \ref{prop_eins}. For the formulation, we say that a symmetric monoidal left adjoint functor $F\colon \cC \to \cC'$ between presentably symmetric monoidal $\infty$-categories is a \emph{localization of a symmetric monoidal left adjoint functor $G\colon \cE \to \cE'$} if there is a commutative diagram of presentably symmetric monoidal $\infty$-categories
\[
\xymatrix@-1pc{
\cE \ar[r]^{G}\ar[d]_L & \cE \ar[d]^{L'} \\
\cC \ar[r]^F & \cC'
}
\]
in which the vertical functors $L$ and $L'$ are symmetric monoidal localizations. It is easy to see that once $G$ and the localizations $L$ and $L'$ are given, $G$ descends to a functor $F$ if and only if it sends local equivalences to local equivalences. Moreover, $F$ is determined up to equivalence by $G$ in that case.

\begin{lemma}\label{lem_presfun}
  Let $F\colon \cC \to \cC'$ be a symmetric monoidal left adjoint functor between presentably symmetric monoidal $\infty$-categories. Then there exists a symmetric monoidal functor $f\colon \cD \to \cD'$ between small symmetric monoidal $\infty$-categories such that $F$ is a localization of the left Kan extension $f_!\colon \cP(\cD) \to \cP(\cD')$.
\end{lemma}
\begin{proof}
  First note that by \cite[Proposition 5.4.7.7]{Lurie_HTT}, every left adjoint functor ${\cC \to \cC'}$ preserves $\kappa$-compact objects for some $\kappa$, i.e., it restricts to a functor ${F\!\!\mid_{\cC^\kappa}\colon \cC^\kappa \to (\cC')^\kappa}$. Since $F$ is left adjoint, it is the left Kan extension of $F\!\!\mid_{\cC^ \kappa}$. This in turn implies that it is a localization of
\[\left({F\!\!\mid_{\cC^\kappa}}\right)_!\colon \cP(\cC^\kappa) \to \cP({\cC'}^\kappa).\]
Replacing the essentially small $\infty$-categories $\cC^\kappa$ and $(\cC')^\kappa$ by small categories proves the claim.
 \end{proof}

In the proof of the next theorem we will use Proposition~\ref{prop_symmon} which we state and prove in Section 4.

\begin{theorem}
  Let $F:\cC \to \cC'$ be a  symmetric monoidal left adjoint functor between presentably symmetric monoidal $\infty$-categories. Then there exist a simplicial symmetric monoidal left adjoint functor $S\colon \cM \to \cM'$ between simplicial, combinatorial and left proper symmetric monoidal model categories $\cM$ and $\cM'$ such that the underlying functor $S_\infty\colon \cM_\infty \to \cM'_\infty$ is equivalent to $F$.
\end{theorem}
\begin{proof}
We first use Lemma \ref{lem_presfun} to conclude that there is a symmetric monoidal functor $f\colon \cD \to \cD'$ between small symmetric monoidal $\infty$-categories such that $F$ is a localization of $f_!$. 
Using Proposition~\ref{prop_symmon} below, we can realize $f_!$ as a left Quillen functor $S\colon \cM \to \cM'$ between symmetric monoidal model categories which model $\cP(\cD)$ and $\cP(\cD')$.  We now equip the categories $\cM$ and $\cM'$ with the local model structures which, by Proposition~\ref{prop_zwei}, correspond to the localization that give $\cC$ and~$\cC'$.  Since the functor $f_!$ descents to a local functor, it preserves local equivalences. Thus the functor $S$ is also left Quillen with respect to the local model structures and the underlying functor of $\infty$-categories represents the functor $F$.
\end{proof}

\section{The contravariant \texorpdfstring{$\cI$}{I}-model structure}\label{sec:contravariant-I}

In this section we set up the model structures that will be used in the proof of Proposition~\ref{prop:monoidal-equivalence} and its functorial refinement Proposition~\ref{prop_symmon}.
\subsection{The contravariant model structure} Let $S$ be a simplicial set and let $\sset/S$ be the category of
objects over $S$. We recall from~\cite[Chapter~2.1.4]{Lurie_HTT}
or~\cite[Section~8]{Joyal-quasi-categories} that $\sset/S$ admits a
\emph{contravariant} model structure where the cofibrations are the
monomorphisms and the fibrant objects $X \to S$ are the \emph{right
  fibrations}, i.e., the maps with the right lifting property with
respect to the set of horn inclusions $\Lambda^{n}_{i}\subseteq
\Delta^n$, $0<i\leq n$. As we will explain in Section~\ref{sec:E-infty}, the contravariant model structure is relevant
for our work because of its connection to presheaf categories coming
from the straightening and unstraightening constructions~\cite[Chapter~2.2.1]{Lurie_HTT}.

We will frequently use
the following feature of the contravariant model structure:  
\begin{lemma}\label{lem:general-Q-adjunction-on-contravariant}\cite[Remark
  2.1.4.12]{Lurie_HTT} A morphism of simplicial sets $S \to T$ induces
  a Quillen adjunction $\sset/S \rightleftarrows \sset/T$ with respect
  to the contravariant model structures.  If $S \to T$ is a Joyal
  equivalence of simplicial sets, then this adjunction is a Quillen
  equivalence.\qed
\end{lemma} 
For simplicial sets $K$ and $T$, we consider the functor
\begin{equation}\label{eq:functor-Ktimes}
  K \times - \colon \sset/T \to \sset/K\times T 
\end{equation}
sending objects and morphisms in $\sset/T$ to their product with
$\id_K$.
\begin{lemma}\label{lem:Ktimes-preserves-acyclic-cof}
  If $f \colon X \to Y$ is an acyclic cofibration in the contravariant
  model structure on $\sset/T$, then $K \times f$ is an acyclic
  cofibration in the contravariant model structure on $\sset/K\times
  T$.
\end{lemma}
We note that since we do not view $K \times - $ as an endofunctor of
$\sset/T$ by projecting away from $K$, this lemma is not implied by
the fact that the contravariant model structure is simplicial.
\begin{proof}[Proof of Lemma~\ref{lem:Ktimes-preserves-acyclic-cof}]
  By~\cite[Lemma 8.16]{Joyal-quasi-categories}, the acyclic
  cofibrations in the contravariant model structure are characterized
  by the left lifting property with respect to the right fibrations
  between objects that are right fibrations relative to the
  base. Hence we have to prove that for every acyclic cofibration $U
  \to V$ in the contravariant model structure on $\sset/T$ and for
  every commutative diagram
  \[\xymatrix@-1pc{
    K\times U \ar[r] \ar[d] & X \ar[d] \\
    K\times V \ar[r] \ar[d] & Y \ar[d] \\
    K\times T \ar[r]^{=} & K\times T }\] in $\sset$ where the right
  hand vertical maps are right fibrations, the upper square admits a
  lift $K\times V \to X$.  Using the tensor/cotensor adjunction
  $(K\times -, (-)^{K})$ on $\sset$, this is equivalent to finding a
  lift in the upper left hand square in
  \[\xymatrix@-1pc{
    U \ar[r] \ar[d] & T \times_{(K\times T)^K} X^K \ar[r] \ar[d] & X^{K} \ar[d] \\
    V \ar[r] \ar[d] & T \times_{(K\times T)^K} Y^K \ar[r] \ar[d] & Y^{K} \ar[d] \\
    T \ar[r]^{=} & T \ar[r]& (K\times T)^{K}\ . }\] Since base change
  preserves right fibrations and the cotensor preserves right
  fibrations (by the dual of~\cite[Corollary 2.1.2.9]{Lurie_HTT}), the
  upper vertical map in the middle is a right fibration between right
  fibrations relative to $T$.
\end{proof}
Since $K\times -$ preserves contravariant cofibrations and all objects
in $\sset/T$ are cofibrant, Ken Brown's lemma and the preceding
statement imply:
\begin{corollary}\label{cor:Ktimes-preserves-contravariant-we}
  The functor $K\times -\colon \sset/T \to \sset/K\times T $ preserves
  contravariant weak equivalences. \qed
\end{corollary}

\subsection{The Joyal \texorpdfstring{$\cI$}{I}-model structure}
Let $\cI$ be the category with the finite sets
$\bld{m}=\{1,\dots,m\}$, $ m\geq 0$, as objects and the injective maps
as morphisms.  An object $\bld{m}$ of $\cI$ is \emph{positive} if
$|\bld{m}| \geq 1$, and $\cI_+$ denotes the full subcategory of $\cI$
spanned by the positive objects.

In the following, we briefly summarize the main results
about the \emph{Joyal $\cI$-model structures} on the functor category
$\ssetI = \mathrm{Fun}(\cI,\sset)$ of $\cI$-diagrams of simplicial
sets from~\cite{Kodjabachev-S_Joyal-I}. These results are motivated by
(and largely derived from) the construction of the corresponding Kan
model structures on $\ssetI$ in~\cite{Sagave-S_diagram}.

We say that a morphism $f$ in $\ssetI$ is a \emph{Joyal
  $\cI$-equivalence} if $\hocolim_{\cI} f$ is a Joyal equivalence in
$\sset$.  It is shown in~\cite[Proposition 2.3]{Kodjabachev-S_Joyal-I}
that $\ssetI$ admits an \emph{absolute} and a \emph{positive} Joyal
$\cI$-model structure.  In both cases, the weak equivalences are the
Joyal $\cI$-equivalences. An object $X$ is fibrant in the absolute
(resp. positive) model structure if each $\alpha \colon \bld{m} \to
\bld{n}$ in $\cI$ (resp. in $\cI_+$) induces a weak equivalence of
fibrant objects $\alpha_*\colon X(\bld{m}) \to X(\bld{n})$ in
$\sset_{\mathrm{Joyal}}$.  In both cases, the $\cI$-model structures
arise as left Bousfield localizations of absolute or positive Joyal
level model structures. Particularly, we will use that a Joyal
$\cI$-equivalence between positive $\cI$-fibrant objects $X\to Y$ is a
\emph{positive Joyal level equivalence}, i.e., $X(\bld{m}) \to
Y(\bld{m})$ is a Joyal equivalence for all $\bld{m}$ in
$\cI_+$. Finally, we note that by~\cite[Corollary
2.4]{Kodjabachev-S_Joyal-I}, there are Quillen equivalences
\begin{equation}\label{eq:pos-abs-Joyal-chain}
\xymatrix@-1pc{\ssetI_{\mathrm{pos}} \ar@<.14pc>[rr]^{\mathrm{id}} && \ssetI_{\mathrm{abs}}  \ar@<.14pc>[ll]^{\mathrm{id}}\ar@<.14pc>[rr]^-{\colim_{\cI}} && \sset_{\mathrm{Joyal}}\ . \ar@<.14pc>[ll]^-{\mathrm{const}_{\cI}}}
\end{equation}

Concatenation of finite ordered sets induces a permutative monoidal
structure on $\cI$ with monoidal unit $\bld{0}$ and symmetry
isomorphism the obvious block permutation. The functor category
$\ssetI$ inherits a symmetric monoidal Day type convolution product
$\boxtimes$ with monoidal unit $\cI(\bld{0},-)$ from the cartesian
product in $\sset$ and the concatenation in $\cI$. Since $\ssetI$ is
tensored over $\sset$, any operad $\cD$ in $\sset$ gives rise to a
category $\ssetI[\cD]$ of $\cD$-algebras in $\ssetI$. The central
feature of the positive model structure on $\ssetI$ is that without
additional assumptions on~$\cD$, the forgetful functor $\ssetI[\cD]
\to \ssetI_{\mathrm{pos}}$ creates a \emph{positive} model structure
on $\ssetI[\cD]$ where a map is weak equivalence or fibration if the
underlying map in $\ssetI_{\mathrm{pos}}$ is~\cite[Theorem
3.1]{Kodjabachev-S_Joyal-I}.

We say that an operad $\cE$ in $\sset$ is an $E_{\infty}$ operad in
$\sset_{\mathrm{Joyal}}$ if $\Sigma_n$ acts freely on the $n$-th space
$\cE(n)$ and $\cE(n) \to *$ is a Joyal equivalence. If $\cE$ is an
$E_{\infty}$ operad in $\sset_{\mathrm{Joyal}}$, then the Joyal model
structure on $\sset$ lifts to a Joyal model structure on $\sset[\cE]$
by an argument analogous to the absolute case of~\cite[Theorem 3.1]{Kodjabachev-S_Joyal-I}.
\begin{theorem}\cite[Theorem 1.2]{Kodjabachev-S_Joyal-I}\label{thm:Joyal-I-com-E-infty-equiv} Let $\cE$ be an $E_{\infty}$ operad in
  $\sset_{\mathrm{Joyal}}$. Then the canonical morphism $\Phi\colon
  \cE \to \cC$ to the commutativity operad and the composite
  adjunction in~\eqref{eq:pos-abs-Joyal-chain} induce a chain of
  Quillen equivalences
\[
\xymatrix@-1pc{\ssetI_{\mathrm{pos}}[\cC] \ar@<-.14pc>[rr]_{\Phi^*}&& \ssetI_{\mathrm{pos}}[\cE] \ar@<.14pc>[rr]^{\colim_{\cI}}\ar@<-.14pc>[ll]_{\Phi_*} &&  \sset_{\mathrm{Joyal}}[\cE]\ . \ar@<.14pc>[ll]^-{\mathrm{const}_{\cI}}}
\]
\end{theorem}
The theorem leads to the following rigidification of $E_{\infty}$
objects in $\sset_{\mathrm{Joyal}}$ to $\cC$-algebras in $\ssetI$,
that is, to commutative monoids in
$(\ssetI,\boxtimes)$.
\begin{corollary}\label{cor:Einfty-rigidification}
  Let $M$ be an $\cE$-algebra in $\sset_{\mathrm{Joyal}}$. There exists a
  rigidification functor $(-)^{\mathrm{rig}}\colon \ssetI[\cE]\to
  \ssetI[\cC]$ and a natural chain of positive Joyal level
  equivalences between positive fibrant objects $\Phi^*(M^{\mathrm{rig}}) \ot
  M^c \to \const_{\cI}M$ in $\ssetI[\cE]$. 
\end{corollary}
\begin{proof}
  This is analogous to the result about $E_{\infty}$ spaces
  in~\cite[Corollary 3.7]{Sagave-S_diagram}: We let $\xymatrix@1{M^c
    \ar@{->>}[r]^-{\sim} & \const_{\cI}M}$ be a cofibrant replacement
  in $\ssetI_{\mathrm{pos}}[\cE]$. Moreover, we let $\Phi_*(M^c) \to
  \Phi_*(M^c)^{\mathrm{fib}}$ be a fibrant replacement in
  $\ssetI_{\mathrm{pos}}[\cC]$. Then the adjunction unit induces an
  $\cI$-equivalence $M^c \to \Phi^*(\Phi_*(M^c)^{\mathrm{fib}})$. Since both objects are positive $\cI$-fibrant, it is even a
  positive Joyal level equivalence. Hence $M^{\mathrm{rig}} =
  \Phi_*(M^c)^{\mathrm{fib}}$ has the desired property.
\end{proof}

\subsection{The contravariant level and \texorpdfstring{$\cI$}{I}-model structures}
Let $Z\colon \cI \to \sset$ be an $\cI$-diagram of simplicial sets. We are interested in various model structures on the comma category $\ssetI/Z$ of objects over $Z$ that are induced from the contravariant model structure. For this purpose, it is important to note that the category $\ssetI/Z$ can be obtained by assembling the comma categories $\sset/Z(\bld{m})$ for varying $\bld{m}$. Indeed, every morphism $\alpha\colon \bld{m}\to\bld{n}$ in $\cI$ induces an adjunction 
\begin{equation}\label{eq:adjunction-from-alpha}
\alpha_{!}\colon \sset/Z(\bld{m}) \rightleftarrows \sset/Z(\bld{n}) \colon \alpha^*
\end{equation}
via composition with and base change along $\alpha_*\colon Z(\bld{m})
\to Z(\bld{n})$, and the adjunctions are compatible with the
composition in $\cI$. We also note that for every object $\bld{m}$
of~$\cI$, there is an adjunction
\begin{equation}\label{eq:adjunction-free-ev}
F_{\bld{m}} \colon  \sset/Z(\bld{m}) \rightleftarrows \ssetI/Z \colon \mathrm{Ev}_{\bld{m}}
\end{equation}
with right adjoint $\mathrm{Ev}_{\bld{m}}(X \to Z) = X(\bld{m}) \to
Z(\bld{m})$ and left adjoint
\[
F_{\bld{m}} (K \to Z(\bld{m})) = \left(\bld{n}\longmapsto \textstyle\coprod_{(\alpha \colon \bld{m}\to\bld{n}) \in \cI}\alpha_{!}(K\to Z(\bld{m}))\right).
\]

A morphism $X \to Y$ in $\ssetI/Z$ is defined to be 
\begin{itemize}
\item an absolute (resp. positive) contravariant level equivalence if for each object (resp. each positive object) $\bld{m}$ of $\cI$, the morphism $X(\bld{m})\to Y(\bld{m})$ is a contravariant weak equivalence in $\sset/Z(\bld{m})$, 
\item an absolute (resp. positive) contravariant level fibration if for each object (resp. each positive object) $\bld{m}$ of $\cI$, the morphism $X(\bld{m})\to Y(\bld{m})$ is a fibration in the contravariant model structure on $\sset/Z(\bld{m})$, and
\item an absolute (resp. positive) contravariant cofibration if it has the left lifting property with respect to all morphisms that are absolute (resp. positive) contravariant level fibrations and equivalences.  
\end{itemize}
\begin{lemma}
These classes of maps define an \emph{absolute} (resp. a \emph{positive}) \emph{contravariant level model structure} on $\ssetI/Z$ which is  simplicial, combinatorial, tractable and left proper. \end{lemma}
\begin{proof}
  The key observation is that by
  Lemma~\ref{lem:general-Q-adjunction-on-contravariant}, the
  adjunction~\eqref{eq:adjunction-from-alpha} is a Quillen adjunction
  with respect to the contravariant model structures. With this
  observation, the existence of the absolute contravariant level model
  structure follows by a standard lifting argument using the
  adjunction
  \[ 
  \textstyle\prod_{\bld{m} \in \cI}\sset/Z(\bld{m}) \rightleftarrows
  \ssetI/Z
  \]
  induced by the adjunctions $(F_{\bld{m}},\mathrm{Ev}_{\bld{m}})$
  from~\eqref{eq:adjunction-free-ev} and the product model structure
  on the codomain; compare~\cite[Theorem 2.28]{barwick_left-right}. If
  $I_{Z(\bld{m})}$ is a set of generating cofibrations for
  $\sset/Z(\bld{m})$, then $\{ F_{\bld{m}}(i) \,|\, \bld{m} \in \cI,
  i\in I_{Z(\bld{m}}\}$ is a set of generating cofibrations for the
  absolute contravariant level model structure, and similarly for the
  generating acyclic cofibrations. The model structure is obviously
  tractable, and it is simplicial and left proper since
  $\sset/Z(\bld{m})$ is.

  In the positive case, we index the above product by the objects of
  $\cI_+$ instead.
\end{proof}

The contravariant model structure on $\sset/Z(\bld{m})$ is cofibrantly
generated and left proper. Since its cofibrations are the
monomorphisms, we may use \[I_{Z(\bld{m})}=\{(K \to Z(\bld{m})) \to
(L\to Z(\bld{m})) \;|\; (K \to L) = (\partial\Delta^n \hookrightarrow
\Delta^n) \}\] as a set of generating cofibrations of
$\sset/Z(\bld{m})$.  Let $W_{Z(\bld{m})}$ be the set of objects in
$\sset/Z(\bld{m})$ given by the domains and codomains of
$I_{Z(\bld{m})}$. By~\cite[Proposition A.5]{Dugger_replacing}, a map
$U\to V$ of fibrant objects in the contravariant model structure on
$\sset/Z(\bld{m})$ is a contravariant weak equivalence if and only if
the induced morphism of simplicial mapping spaces $
\Map_{Z(\bld{m})}(K,U) \to \Map_{Z(\bld{m})}(K,V) $ is a weak homotopy
equivalence of simplicial sets for every object $K\to Z(\bld{m})$ in
$W_Z(\bld{m})$.  For an object $K \to Z(\bld{m})$ in $W_Z(\bld{m})$
and a morphism $\alpha \colon \bld{m} \to \bld{n}$ in $\cI$, we let
\[ 
F_{\bld{n}}(\alpha_{!}(K)) \to F_{\bld{m}}(K)
\] 
be the morphism in $\ssetI/Z$ that is adjoint to the inclusion
\[
\alpha_{!}(K) \hookrightarrow \textstyle\coprod_{(\beta \colon \bld{m}\to\bld{n}) \in \cI}\beta_{!}(K) = \mathrm{Ev}_{\bld{n}}(F_{\bld{m}}(K))
\]
of the summand indexed by $\alpha$. We write 
\begin{equation}
S^Z = \{ F_{\bld{n}}(\alpha_{!}(A)) \to F_{\bld{m}}(A) \;|\; (\alpha\colon \bld{m}\to\bld{n}) \in \cI,\; (A\to Z(\bld{m})) \in W_Z(\bld{m})\}
\end{equation}
for the set of all such maps and let $S^Z_{+}$ be the subset of $S^Z$
consisting those maps that come from $\alpha \in \cI_+$. 

\begin{proposition}\label{prop:contravariant-I}
  The left Bousfield localization of the absolute (resp. positive)
  contravariant level model structure on $\ssetI/Z$ with respect to
  $S^Z$ (resp. $S^Z_+$) exists. It is a simplicial, combinatorial,
  tractable and left proper model structure.\qed
\end{proposition}

We refer to this model structure as the \emph{absolute}
(resp. \emph{positive}) \emph{contravariant $\cI$-model
  structure}. The weak equivalences in these model structures are
called \emph{absolute} (resp. \emph{positive}) $\cI$-equivalences. The
cofibrations are the same as in the respective level model structures.
An object $X \to Z$ is absolute (resp. positive) contravariant
$\cI$-fibrant if is absolute (resp. positive) contravariant level
fibrant an each $\alpha\colon \bld{m}\to\bld{n}$ in $\cI$ (resp. in
$\cI_+$) induces a contravariant weak equivalence $X(\bld{m})\to
\alpha^*(X(\bld{n}))$ in $\sset/Z(\bld{m})$.

The contravariant $\cI$-model structures are homotopy invariant in
level equivalences of the base:
\begin{lemma}\label{lem:hty-invariance-in-level-equiv}
  Let $Z \to Z'$ be a morphism in $\ssetI$. Then the induced adjunction
  $\ssetI/Z \rightleftarrows \ssetI/Z'$ is a Quillen adjunction with
  respect to the absolute and positive contravariant $\cI$-model
  structures. If  $Z \to Z'$
  is an absolute (resp. a positive) Joyal level equivalence, then
  it is a Quillen equivalence with respect to the absolute
  (resp. positive) contravariant $\cI$-model structures. 
\end{lemma}
\begin{proof}
  We treat the absolute case, the positive case is similar. It is
  clear that the adjunction in question is a Quillen adjunction with
  respect to the absolute level model structure. Since $(Z \to
  Z')_{!}(S_Z)$ is a subset of $S_{Z'}$, there is an induced Quillen
  adjunction on the localizations.  Using
  Lemma~\ref{lem:general-Q-adjunction-on-contravariant}, it is also
  clear that an absolute Joyal level equivalence induces a Quillen
  equivalence with respect to the absolute contravariant level model
  structures. To see that it is a Quillen equivalence, we note that by
  adjunction, the $(Z \to Z')_{!}(S_Z)$-local objects coincide with
  the $S_{Z'}$-local objects. 
\end{proof}

We write $(-)_{\cI} = \colim_{\cI}$ for the colimit over $\cI$ and
note that the adjunction $(-)_{\cI}\colon \ssetI \rightleftarrows
\sset \colon \const_{\cI}$ induces adjunctions of overcategories
\begin{equation}\label{eq:colimI-constI-chain}
\ssetI/Z \rightleftarrows \ssetI/(\const_{\cI}(Z_{\cI})) \rightleftarrows \sset/Z_{\cI}.
\end{equation}

\begin{lemma}\label{lem:colimI-constI-abs-Q-equiv}
  Let $Z$ be cofibrant and fibrant in the absolute Joyal $\cI$-model
  structure on $\ssetI$. Then the composite adjunction 
$\ssetI/Z  \rightleftarrows \sset/Z_{\cI}$
 is a Quillen equivalence with
  respect to the absolute contravariant $\cI$-model structure on
  $\ssetI/Z$ and the contravariant model structure on $\sset/Z_{\cI}$.
\end{lemma}
\begin{proof}
  Since $Z$ is cofibrant and fibrant, the Quillen equivalence~\eqref{eq:pos-abs-Joyal-chain} shows that the adjunction unit $Z \to
  \const_{\cI}(Z_{\cI})$ is an absolute Joyal level equivalence. Hence the
  first adjunction in~\eqref{eq:colimI-constI-chain} is a Quillen
  equivalence by Lemma~\ref{lem:hty-invariance-in-level-equiv}. It
  follows from the definitions that the second adjunction is a Quillen
  adjunction whose right adjoint detects weak equivalences between
  fibrant objects. Hence it is sufficient to show that the derived
  adjunction unit is an absolute contravariant $\cI$-equivalence. Let
  $X \to \const_{\cI}(Z_{\cI})$ be a cofibrant object in the absolute
  contravariant $\cI$-model structure. A fibrant replacement $X \to
  X'$ and the adjunction counit of
  $(F_{\bld{0}},\mathrm{Ev}_{\bld{0}})$ provide a chain of absolute
  contravariant $\cI$-equivalences between cofibrant objects
  \[\xymatrix{X \ar@{>->}[r]^{\sim}& X ' &F_{\bld{0}}
    \mathrm{Ev}_{\bld{0}}(X'). \ar[l]_-{\sim}}\]
  Since $\bld{0}$ is initial in $\cI$, there is an isomorphism
  $F_{\bld{0}} \mathrm{Ev}_{\bld{0}}(X') \iso \const_{\cI}
  X'(\bld{0})$. The claim follows because the evaluation of the
  adjunction unit of $((-)_{\cI},\const_{\cI})$ on $\const_{\cI}
  X'(\bld{0})$ is even an isomorphism and $\const_{\cI}$ preserves
  weak equivalence between all objects.
\end{proof}

\begin{proposition}\label{prop:pos-abs-Q-equiv}
  For every absolute Joyal $\cI$-fibrant $Z$ in $\ssetI$, the
  identity functors form a Quillen equivalence
  $(\ssetI/Z)_{\mathrm{pos}} \rightleftarrows (\ssetI/Z)_{\mathrm{abs}}$
  with respect to the positive and absolute contravariant $\cI$-model
  structures.
\end{proposition}
\begin{proof}
  Let $Z^c \to Z$ be a cofibrant replacement in the absolute Joyal
  $\cI$-model structure and let $Z^c \to \const_{\cI}(Z^c_{\cI})$ be
  the adjunction unit. Since these two maps are absolute Joyal level
  equivalences, Lemma~\ref{lem:hty-invariance-in-level-equiv} and the
  two out of three property for Quillen equivalences reduce the claim
  to the case where $Z = \const_{\cI}T$ for a simplicial set $T$. 

  The category $\ssetI/(\const_{\cI}T)$ is equivalent to the category
  $(\sset/T)^{\cI}$ of $\cI$-diagrams in $\sset/T$. Under this
  equivalence, the absolute contravariant $\cI$-model structure
  corresponds to the homotopy colimit model structure on
  $(\sset/T)^{\cI}$ provided by~\cite[Theorem
  5.1]{Dugger_replacing}. The cited theorem implies that the weak
  equivalences in the absolute contravariant $\cI$-model structure are
  the maps that induce contravariant weak equivalences under
  $\hocolim_{\cI}\colon (\sset/T)^{\cI} \to \sset/T$.

  The argument for comparing the model structures now works as
  in~\cite[Proposition 2.3]{Kodjabachev-S_Joyal-I}: The inclusion
  $\cI_+ \to \cI$ is homotopy cofinal~\cite[Proof of Corollary
  5.9]{Sagave-S_diagram}, and hence every positive contravariant level
  equivalence is an $\hocolim_{\cI}$-equivalence. Together with
  $S^{\const_{\cI}\!T}_{+} \subset S^{\const_{\cI}\!T}$, this shows
  that every positive contravariant $\cI$-equivalence is an absolute
  contravariant $\cI$-equivalence. For the converse, it suffices to
  show that a $\hocolim_{\cI}$-equivalence of positive contravariant
  $\cI$-fibrant objects is a positive contravariant $\cI$-equivalence. Using
  again that $\cI_+ \to \cI$ is homotopy cofinal, this follows by
  restricting along $\cI_+ \to \cI$ and applying~\cite[Theorem
  5.1(a)]{Dugger_replacing} in $(\sset/T)^{\cI_+}$.
\end{proof}

\begin{corollary}\label{cor:colimI-constI-pos-Q-equiv}
If $Z$ is absolute Joyal cofibrant and positive Joyal $\cI$-fibrant, then
$\ssetI/Z  \rightleftarrows \sset/Z_{\cI}$
is a Quillen equivalence with respect to the positive contravariant
$\cI$-model structure on $\ssetI/Z$ and the contravariant model
structure on $\sset/Z_{\cI}$.
\end{corollary}
\begin{proof}
  Since the derived adjunction unit $Z \to
  \const_{\cI}((Z_{\cI})^{\mathrm{Joyal-fib}}) = Z'$ is a positive level
  equivalence, the adjunction $\ssetI/Z \rightleftarrows
  \ssetI/Z' $ is a Quillen equivalence with
  respect to the positive contravariant $\cI$-model structure by
  Lemma~\ref{lem:hty-invariance-in-level-equiv}. Because
  $Z'$ is cofibrant and
  fibrant in the absolute Joyal $\cI$-model structure,
  Proposition~\ref{prop:pos-abs-Q-equiv} and
  Lemma~\ref{lem:colimI-constI-abs-Q-equiv} show the claim.
\end{proof}

Let $N$ be a commutative monoid object in $(\ssetI,\boxtimes)$. Then
the overcategory $\ssetI/N$ inherits a symmetric monoidal product 
\[(X \to N) \boxtimes (Y \to N) = (X\boxtimes Y \to N\boxtimes N \to
N) \] from the symmetric monoidal structure of $N$ and the
multiplication of $N$. 

The following result is a key step in the proof of our main result. 
\begin{theorem}\label{thm_covariant}
  Let $\cE$ be an $E_{\infty}$ operad in $\sset_{\mathrm{Joyal}}$ and
  let $M$ be an $\cE$-algebra.  Then there is a chain of
  Quillen equivalences of simplicial, combinatorial and left proper
  model categories
\[
  \ssetI/M^{\mathrm{rig}} \leftrightarrows \ssetI/M^{c} \rightleftarrows  \ssetI/\const_{\cI}M \rightleftarrows \sset/M
\]
relating $\sset/M$ with the contravariant model structure and the
symmetric monoidal model category $\ssetI/M^{\mathrm{rig}}$ with the
positive contravariant $\cI$-model structure. The chain is natural
with respect to $M$.
\end{theorem}
\begin{proof}
  Using the chain of positive level equivalences $M^{\mathrm{rig}} \ot
  M^c \to \const_{\cI}M$ from
  Corollary~\ref{cor:Einfty-rigidification} and the fact that
  $\const_{\cI}M \iso F_{\bld{0}}M$ is absolute Joyal $\cI$-cofibrant,
  the chain of Quillen equivalences is a consequence of
  Lemma~\ref{lem:hty-invariance-in-level-equiv} and
  Corollary~\ref{cor:colimI-constI-pos-Q-equiv}. It is shown in
  Corollary~\ref{cor:contravariant-I-monoidal} that $
  \ssetI/M^{\mathrm{rig}}$ satisfies the pushout product axiom.
\end{proof}

We need one more observation about the tensor product on $\ssetI/M^{\mathrm{rig}}$. We call an object in $\Ho(\ssetI/M^{\mathrm{rig}})$ representable if it corresponds 
to an object of the form $\Delta^0 \to M$ under the equivalence $\Ho(\ssetI/M^{\mathrm{rig}}) \simeq \Ho(\sset/M)$ induced by the chain of Quillen equivalences from  Theorem \ref{thm_covariant}. Note that these are 
precisely the objects which correspond to representable presheaves under the equivalence to presheaves on the $\infty$-category $M$.
\begin{lemma}\label{lemma_rep}
The tensor product of two representables in $\Ho(\ssetI/M^{\mathrm{rig}})$ is again representable.
\end{lemma}
\begin{proof}
It follows from the construction of $M^{\mathrm{rig}}$ and the chain of Quillen equivalences that the representables in $\Ho(\ssetI/M^{\mathrm{rig}})$ are represented by the cofibrant objects of the form 
$F_{\bld{k}}^{\cI}(\Delta^0) \to M$ with $\bld{k}$ an positive object of $\cI$. Since $F_{\bld{k}}^{\cI}(K)\boxtimes F_{\bld{l}}^{\cI}(L) \iso F_{\bld{k}\sqcup\bld{l}}^{\cI}(K\times L)$, 
this set of objects is closed under the monoidal product. 
\end{proof}

\subsection{Monoidal properties of the contravariant
  \texorpdfstring{$\cI$}{I}-model structure}
 The following proposition is the key tool for the homotopical analysis of the $\boxtimes$-product on $\ssetI/N$ for a commutative $N$. Both its statement and proof are analogous to~\cite[Proposition 8.2]{Sagave-S_diagram} and~\cite[Proposition 2.6]{Kodjabachev-S_Joyal-I}:

\begin{proposition}\label{prop:X-boxtimes-preserves}
  Let $N$ be a commutative monoid object in $\ssetI$. If $X \to N$ is
  absolute contravariant cofibrant, then $X\boxtimes - \colon \ssetI/N
  \to \ssetI/N$ preserves positive contravariant $\cI$-equivalences
  between arbitrary objects.
\end{proposition}
\begin{proof}
  We begin by showing that if $Y_1 \to Y_2$ is an absolute
  contravariant level equivalence in $\ssetI/N$, then so is
  $X\boxtimes Y_1 \to X\boxtimes Y_2$. For this, we use a cell
  induction argument and first consider the case $X = F_{\bld{m}}(K)$.

  By~\cite[Lemma 5.6]{Sagave-S_diagram}, the map
  $(F_{\bld{m}}(K)\boxtimes (Y_1 \to Y_2))(\bld{n})$ is isomorphic to
  \begin{equation}\label{eq:FmK-boxtimes-decomposition}
    K \times (\colim_{\bld{m}\sqcup \bld{k}\to \bld{n}} 
    (Y_1(\bld{k}) \to   Y_2(\bld{k}))) 
  \end{equation} 
  where the colimit is taken over the comma category $(\bld{m}\sqcup -
  \downarrow \bld{n})$. Since each connected component of this comma
  category has a terminal object, we can choose a set $A$ of 
  morphisms $\alpha \colon \bld{m}\sqcup \bld{k}\to \bld{n}$ such
  that~\eqref{eq:FmK-boxtimes-decomposition} is isomorphic to
  \[
  \textstyle\coprod_{(\alpha \colon \bld{m}\sqcup \bld{k}\to \bld{n}) \in A} K\times (Y_1(\bld{k}) \to   Y_2(\bld{k})).
  \]
  Using Corollary~\ref{cor:Ktimes-preserves-contravariant-we}, it
  follows that each summand is a contravariant weak equivalence in
  $\sset/(K\times N(\bld{k}))$. Composing with the map
  \[K\times N(\bld{k}) \to N(\bld{m}) \times N(\bld{k}) \to N(\bld{n})\]
  induced by the morphism $\alpha \colon
  \bld{k}\sqcup\bld{m}\to\bld{n}$ indexing the summand, it follows
  that each summand is a contravariant weak equivalence in
  $\sset/N(\bld{n})$. Hence~\eqref{eq:FmK-boxtimes-decomposition} is a
  contravariant weak equivalence in $\sset/N(\bld{n})$.

  Next we assume that $F_{\bld{m}}(K)\to F_{\bld{m}}(L)$ is a
  generating cofibration in $\ssetI/N$, that $X_{\alpha+1}$ is the
  pushout of $F_{\bld{m}}(L)\ot F_{\bld{m}}(K)\to X_{\alpha}$ in
  $\ssetI/N$ and that $X_{\alpha}\boxtimes -$ preserves absolute
  contravariant level equivalences. By the above decomposition,
  $F_{\bld{m}}(K \to L)\boxtimes Y_i$ is a cofibration when evaluated
  at $\bld{n}$, and the gluing lemma in the left proper model category
  $\sset/N(\bld{n})$ shows that $X_{\alpha +1}\boxtimes (Y_1 \to Y_2)$
  is an absolute contravariant level equivalence in $\sset/N$. Since a
  general absolute contravariant cofibrant object $X$ is a retract of
  a colimit of a sequence of maps of this form, it follows that
  $X\boxtimes -$ preserves absolute contravariant level equivalences.

  We now turn to the statement of the proposition and assume that $Y_1
  \to Y_2$ is a positive contravariant $\cI$-equivalence. By applying
  the previous argument to cofibrant replacements of the $Y_i$, we may
  assume that the $Y_i$ are absolute contravariant cofibrant. Let
  $\xymatrix@1{Y_2 \ar@{ >->}[r]& N^c \ar@{->>}[r]^{\sim}& N}$ be a
  factorization in the absolute Joyal model structure. By
  Lemma~\ref{lem:hty-invariance-in-level-equiv}, $Y_1\to Y_2$ is a
  positive contravariant $\cI$-equivalence in $\ssetI/N^c$. Since the
  induced map of colimits is a contravariant equivalence in
  $\sset/(X_{\cI}\times N^c_{\cI})$ by
  Corollaries~\ref{cor:Ktimes-preserves-contravariant-we}
  and~\ref{cor:colimI-constI-pos-Q-equiv}, another application of
  Corollary~\ref{cor:colimI-constI-pos-Q-equiv} shows that the induced
  map $X \boxtimes Y_1 \to X\boxtimes Y_2$ is a positive contravariant
  $\cI$-equivalence in $\ssetI/(X\boxtimes N^c)$. Composing with
  $X\boxtimes N^c \to N \boxtimes N \to N$ shows that $X \boxtimes Y_1
  \to X\boxtimes Y_2$ is a positive contravariant $\cI$-equivalence in
  $\ssetI/N$.
\end{proof}

\begin{corollary}\label{cor:contravariant-I-monoidal}
Let $N$ be a commutative monoid object in $\ssetI$. The positive contravariant $\cI$-model structure on $\ssetI/N$ satisfies the pushout product axiom and the monoid axiom as defined in~\cite{Schwede_S-algebras}.  
\end{corollary}
\begin{proof}
  The cofibration part of the pushout product axiom follows
  from~\cite[Proposition 8.4]{Sagave-S_diagram}. As explained
  in~\cite[Proposition 8.4]{Sagave-S_diagram},
  Proposition~\ref{prop:X-boxtimes-preserves} implies the statement
  about the generating acyclic cofibrations.

  For the monoid axiom, we have to show that transfinite composition
  of cobase changes of maps of the form $X\boxtimes (Y_1 \to Y_2)$
  with $Y_1\to Y_2$ an acyclic cofibration are contravariant
  $\cI$-equivalences.  Since $\ssetI/N$ is tractable, we may assume
  that also the generating acyclic cofibrations of the positive
  contravariant $\cI$-model structure have cofibrant domains and
  codomains~\cite[Corollary~2.8]{barwick_left-right}. Using
  Proposition~\ref{prop:X-boxtimes-preserves} and a cofibrant
  replacement of $X$, it follows that $X\boxtimes (Y_1 \to Y_2)$ is a
  contravariant $\cI$-equivalence. It is also an injective level
  cofibration, i.e., a cofibration when evaluated at any object
  $\bld{n}$ of $\cI$. Using a cofibrant replacement in the absolute
  contravariant level model structure, it follows that cobase changes
  and transfinite compositions preserve morphisms that are both
  contravariant $\cI$-equivalences and injective level cofibrations.
\end{proof}
The next proposition states that (any monoidal left Bousfield
localization of) the positive contravariant $\cI$-model structure on $
\ssetI/N$ lifts to operad algebras in the best possible way.
\begin{proposition}\label{prop:operad-lift}
Let $N$ be a commutative monoid object in $\ssetI$, let $\cM$ be
a left Bousfield localization of the positive contravariant $\cI$-model
structure on $ \ssetI/N$, and assume that $\cM$ satisfies the pushout
product axiom with respect to $\boxtimes$. 
\begin{enumerate}[(i)]
\item For any operad $\mathcal O$ in $\sset$, the forgetful functor
$\cM[\mathcal O] \to \cM$ from the category of $\cO$-algebras in $\cM$
creates a model structure on $\cM[\mathcal O]$. 
\item If $\cP \to \cO$ is a weak equivalence of operads,  
then the induced adjunction $\cM[\mathcal P]\rightleftarrows \cM[\mathcal O]$ 
is a Quillen equivalence.
\end{enumerate}
\end{proposition}
\begin{proof}
The criteria given in~\cite[Theorems 5.10 and
7.5]{Pavlov-Scholbach-admissibility} reduce this to showing that $\cM$ is \emph{symmetric $h$-monoidal} and \emph{symmetric flat} in the sense of~\cite[Definitions 4.2.4 and 4.2.7]{Pavlov-Scholbach-symmetric-powers}. 

As a first step, we show that the levelwise cofibrations in $\cM$ are $h$-cofibrations in the sense of~\cite[Definition 2.0.4]{Pavlov-Scholbach-symmetric-powers}, i.e., that cobase change along levelwise cofibrations preserves weak equivalences. For this it is sufficient that pushouts along levelwise cofibrations are are homotopy pushouts in $\cM$. Let $V \ot U \to X$ be a diagram in $\cM$ with $U\to V$ a levelwise cofibration. Let $U \to V'\to V$ be a factorization of $U \to V$ into a positive $\cI$-cofibration $U \to V'$ and a positive level equivalence $V' \to V$. Then the induced map of pushouts $V'\coprod_{U}X \to V\coprod_{U}X$ is a positive level equivalence by a levelwise application of the left properness of the contravariant model structure. Hence $V\coprod_{U}X $ is a homotopy pushout.

By~\cite[Theorem 4.3.9(iii)]{Pavlov-Scholbach-symmetric-powers}, it is
sufficient to verify symmetric $h$-monoidality on the generating
(acyclic) cofibrations. For this we let
\begin{equation}\label{eq:fam-gen-cof}
v_i = F_{\bld{k_i}}^{\cI}(\partial \Delta^{m_i} \to \Delta^{m_i}), \quad 1\leq i\leq e
\end{equation}
be a family of generating cofibrations of $\cM$. (We drop the augmentation to $N$ from the notation.) Let $(n_i)_{1\leq i\leq e}$ be a family of natural numbers. Then the iterated pushout product map 
\begin{equation}\label{eq:iterated-pp}
v = v_1^{\Box n_1} \Box \dots \Box v_{e}^{\Box n_e}
\end{equation}
is a $\Sigma_{(n_i)} = \Sigma_{n_1} \times \dots \times \Sigma_{n_e}$-equivariant map. For every $\Sigma_{(n_i)}$-object $Y$ in $\cM$, there is an isomorphism
\[
Y \boxtimes v \iso (Y \boxtimes F_{\bld{k}}^{\cI}(*)) \times \iota
\]
where $\bld{k} = \bld{k_1}^{\sqcup n_1}\sqcup \dots \sqcup \bld{k_e}^{\sqcup n_e}$ and 
\[
\iota = (\partial \Delta^{m_1} \to \Delta^{m_1})^{\Box n_1}\Box \dots \Box (\partial \Delta^{m_e} \to \Delta^{m_e})^{\Box n_e}
\]
is the iterated pushout product map in spaces. Hence $Y\Box v$ is a levelwise cofibration of simplicial sets, and so is its quotient by the $\Sigma_{(n_i)}$-action. This verifies the cofibration part of the symmetric $h$-monoidality. 

Next let $(v_i\colon V_i \to W_i)_{1\leq i\leq e}$ be a family of
generating acyclic cofibrations for $\cM$. We may assume that the
$V_i$ and $W_i$ are positive cofibrant since $\sset^{\cI}/N$ and hence
$\cM$ is tractable. Let $v\colon V \to W$ be defined as
in~\eqref{eq:iterated-pp} and let $Y$ be a $\Sigma_{(n_i)}$-object in
$\cM$. For the acyclic cofibration part of the symmetric
$h$-monoidality, we have to show that $(Y \boxtimes
v)_{\Sigma_{(n_i)}}$ is a weak equivalence in $\cM$. Let $f \colon X
\to Y$ be a cofibrant replacement in $\cM$ and consider the diagram
\[
\xymatrix@-1pc{X\boxtimes V \ar[rr]^{f \boxtimes V} \ar[d]_{X\boxtimes v} && Y\boxtimes V \ar[d]^{Y\boxtimes v} && (Y\boxtimes V)^{\mathrm{cof}} \ar[d]^{g} \ar@{->>}[ll]^{\sim}_{p_V}\\
X\boxtimes W \ar[rr]^{f \boxtimes W} && Y \boxtimes W&&  (Y\boxtimes W)^{\mathrm{cof}}\ar@{->>}[ll]^{\sim}_{p_W}}
\]
where $g$ is a replacement of $Y\boxtimes v$ by a map of cofibrant
objects in the projective model structure on $\cM^{\Sigma_{(n_i)}}$.
The map $X\boxtimes v$ is a weak equivalence in $\cM$ by the pushout
product axiom in $\cM$, and the maps $f\boxtimes V$ and $f\boxtimes W$
are positive $\cI$-equivalences by
Proposition~\ref{prop:X-boxtimes-preserves}.  Hence $Y \boxtimes v$
and $g$ are weak equivalences in $\cM$. To see that $Y \boxtimes v$
becomes a weak equivalence after taking $\Sigma_{(n_i)}$-orbits, we
first note that $g$ induces a weak equivalence of
$\Sigma_{(n_i)}$-orbits because it is a map of cofibrant
objects. Hence it is sufficient to show that $p_V$ and $p_W$ induce a
weak equivalence of $\Sigma_{(n_i)}$-orbits. Since these are actually
positive contravariant level equivalences, it is sufficient to show
that the $\Sigma_{(n_i)} $-action on ${Y\boxtimes W}$ is free in
positive levels. The group  $\Sigma_{n_i}$-acts freely on $W_i^{\boxtimes n_i}(\bld{m})$ because  $W_i$ is positive cofibrant~\cite[Lemma~2.9]{Kodjabachev-S_Joyal-I}. The fact that there is a morphism  of
$\Sigma_{(n_i)}$-spaces
\[
(Y\boxtimes W)(\bld{m}) \to W(\bld{m}) \to (W_1^{\boxtimes n_1}\boxtimes \dots \boxtimes W_e^{\boxtimes n_e})(\bld{m}) \to W_1^{\boxtimes n_1}(\bld{m})\times \dots \times W_e^{\boxtimes n_e}(\bld{m})
\]
thus implies that  $\Sigma_{(n_i)} $ act freely on ${Y\boxtimes W}(\bld{m})$. This completes the acyclic cofibration part of the  symmetric $h$-monoidality. 

For symmetric flatness, it is by~\cite[Theorem 4.3.9(ii)]{Pavlov-Scholbach-symmetric-powers} sufficient to show that for a weak equivalence $y \colon Y \to Z$ in the projective model structure on $\cM^{\Sigma_{(n_i)}}$ and for $v$ as in~\eqref{eq:fam-gen-cof} and~\eqref{eq:iterated-pp}, the map $(y \Box v)_{\Sigma_{(n_i)}}$ is a weak equivalence in $\cM$. Here $y \Box v$ is the pushout product map in the square
\[
\xymatrix@-1pc{Y\boxtimes V \ar[rr]^{y \boxtimes V} \ar[d]_{Y\boxtimes v} && Z\boxtimes V \ar[d]^{Z\boxtimes v} \\
Y\boxtimes W \ar[rr]^{y \boxtimes W} && Z \boxtimes W.}  
\]
Replacing $y$ by a weak equivalence of cofibrant objects in $\cM^{\Sigma_{(n_i)}}$ and using Proposition~\ref{prop:X-boxtimes-preserves} and the pushout product axiom in $\cM$ shows that the vertical maps are weak equivalences in $\cM$. Since $X \boxtimes v$ is a levelwise cofibration by~\cite[Proposition 7.1(vi)]{Sagave-S_diagram}, it is an $h$-cofibration by the argument at the beginning of the proof. Hence $y\Box v$ is a weak equivalence in $\cM$ by 2-out-of-3. Arguing as in the previous step of the proof, the fact that  $\Sigma_{(n_i)}$ acts freely on the positive levels of $Y \boxtimes W$ implies that  $(y\Box v)_{\Sigma_{(n_i)}}$ is a weak equivalence in $\cM$. 
\end{proof}

\begin{remark}
The argument given in the previous proof actually shows the stronger statement that the two assertions in the proposition hold for colored operads and for operads internal to $\cC$. 
\end{remark}

\section{\texorpdfstring{$E_{\infty}$}{E-infinity} objects and
  symmetric monoidal \texorpdfstring{$\infty$}{infinity}-categories}\label{sec:E-infty}
The goal of this section is to prove Proposition~\ref{prop:monoidal-equivalence} and its functorial refinement Proposition~\ref{prop_symmon}.

The $\infty$-category $\SMC$ of small symmetric monoidal
$\infty$-categories is equivalent to the $\infty$-category
$\mathrm{CAlg}(\mathrm{Cat_{\infty}})$ of commutative algebra objects
in $\infty$-categories~\cite[Remark 2.4.2.6]{Lurie_HA}. Now let $\cE$
be an $E_{\infty}$ operad in $\sset_{\mathrm{Joyal}}$ in the above
sense (for example, the Barratt--Eccles operad). We will use the following
result about the rectification of commutative algebras in the
$\infty$-categorical sense to operad algebras in the model category.
\begin{proposition}\label{prop:SMC-as-localization}
There is an equivalence of $\infty$-categories
\begin{equation}\label{eq:SMC-as-localization}
\left(\sset_{\mathrm{Joyal}}[\cE]\right)_{\infty} \simeq \mathrm{CAlg}(\mathrm{Cat_{\infty}})
\end{equation}
relating the $\infty$-category associated with the model category of $\cE$-algebras in $\sset_{\mathrm{Joyal}} $ and $\mathrm{CAlg}(\mathrm{Cat_{\infty}})$. 
For an object $M$ in $\sset_{\mathrm{Joyal}}[\cE] $, the $\infty$-category represented by~$M$ is naturally equivalent to the underlying $\infty$-category of the associated commutative algebra in~$\mathrm{Cat_{\infty}}$. 
\end{proposition}
\begin{proof}
  This is essentially a consequence
  of~\cite[Theorem~7.10]{Pavlov-Scholbach-admissibility} (which is in
  turn based on~\cite[Theorem 4.5.3.7]{Lurie_HA}). However,
  \cite[Theorem~7.10]{Pavlov-Scholbach-admissibility} is not directly
  applicable since it is formulated in terms of simplicial model
  categories and simplicial operads, while $\cE$ is an operad in
  $\sset_{\mathrm{Joyal}}$. As explained
  in~\cite[Remark~7.12]{Pavlov-Scholbach-admissibility}, this context
  requires a different argument for identifying the free $\cE$-algebra
  $\cE(X)$ on a cofibrant object $X$ with its derived counterpart in
  $\mathrm{CAlg}(\mathrm{Cat_{\infty}})$. To circumvent this problem, we note that under
  the chain of Quillen equivalences in
  Theorem~\ref{thm:Joyal-I-com-E-infty-equiv}, $\cE(X)$ corresponds to
  the free commutative algebra on a positive cofibrant replacement of
  $\const_{\cI}(X)$. Using~\cite[Lemma~2.9]{Kodjabachev-S_Joyal-I} in
  place of~\cite[Lemma 4.5.4.11(3)]{Lurie_HA}, the claim about
  $\cE(X)$ follows as in part (e) of the proof of~\cite[Theorem
  4.5.3.7]{Lurie_HA}.
\end{proof}
We are now ready to give the proof of the key proposition from Section~\ref{sec_two}:
\begin{proof}[Proof of Proposition~\ref{prop:monoidal-equivalence}]
  Using the above discussion, we choose an $\cE$-algebra $M$ in
  $\sset$ representing the given small symmetric monoidal
  $\infty$-category $\cD$ and consider the model category
  $\ssetI/M^{\mathrm{rig}} $ arising from
  Theorem~\ref{thm_covariant}. By
  Proposition~\ref{prop:contravariant-I} and
  Corollary~\ref{cor:contravariant-I-monoidal}, this is a simplicial,
  combinatorial, tractable and left proper symmetric monoidal model
  category.
  Let $\cC = \left(\ssetI/M^{\mathrm{rig}}\right)_{\infty}$ be the
  presentably symmetric monoidal $\infty$-category associated with
  $\ssetI/M^{\mathrm{rig}} $. We will show that $\cC$ and
  $\cP(\cD)$ are equivalent as symmetric monoidal $\infty$-categories.

It is immediate from Theorem \ref{thm_covariant} that after forgetting the monoidal structure, $\cC$ is equivalent to the underlying $\infty$-category of the contravariant model structure on $\sset/M$. 
The underlying $\infty$-category of $\sset/M$ is equivalent to the $\infty$-category $\cP(\cD)$ by means of the $\infty$-categorical Grothendieck construction \cite[Theorem 2.2.1.2]{Lurie_HTT} and the fact that the underlying $\infty$-category of $M$ is equivalent to the underlying
$\infty$-category of~$\cD$. Note that all the involved equivalences, i.e., the equivalences coming from Theorem \ref{thm_covariant} as well as the Grothendieck construction, are pseudonatural in $M$, that is, natural in a 2-categorical sense. 
Thus invoking \cite[Appendix A]{HGN} we conclude that the induced equivalence of $\infty$-categories
\begin{equation*}
\Phi\colon \cP(\cD) \to \cC 
\end{equation*}
is natural in $\cD$ in the $\infty$-categorical sense. Note however that this equivalence does not necessarily need to respect the symmetric monoidal structures. 

We need to show that $\Phi$ is compatible with the symmetric monoidal
structures on $\cP(\cD)$ and $\cC$. By the universal property of the
Day convolution symmetric monoidal structure on $\cD$ reviewed in
Section \ref{sec_two}, it suffices to equip the functor
\begin{equation*}
\Psi = \Phi \circ j\colon \cD \to \cC 
\end{equation*}
given by composition with the Yoneda embedding $j: \cD \to \cP(\cD)$
with a symmetric monoidal structure. The functor $\Psi$ is also
natural in $\cD$ in the $\infty$-categorical sense.  We denote the
essential image of $\Psi$ by $\Psi(\cD)$. By construction $\Psi(\cD)$
is a full subcategory of~$\cC$. It follows from Lemma \ref{lemma_rep}
that $\Psi(\cD)$ is closed under tensor products in $\cC$.  Thus it
inherits a symmetric monoidal structure from $\cC$ such that the
inclusion functor $\Psi(\cD) \to \cC$ is a symmetric monoidal functor.

To complete the proof, it is sufficient to show that the corestriction $\cD \to \Psi(\cD)$ of $\Psi$ is a symmetric monoidal functor. For this we use the equivalence~\eqref{eq:SMC-as-localization} and the functoriality of the involved constructions to view the construction $\cD \mapsto \Psi(\cD)$ as a functor 
\[
G\colon \SMC\to \SMC
\]
This functor $G$ comes with a natural equivalence $UG \simeq U$ given by $\Psi$ 
where $U\colon \SMC \to \Cat_\infty$ is the canonical forgetful functor. The next lemma implies that $G$ is canonically equivalent to the identity functor on $\SMC$ and that the equivalence refines $\Psi$. 
We conclude that for each~$\cD$, the functor $\Psi$ refines to an equivalence $\cD \simeq \Psi(\cD)$ of symmetric monoidal $\infty$-categories. 
\end{proof}

\begin{lemma}\label{lastlemmahere}
Let $G\colon \SMC \to \SMC$ be a functor together with an equivalence $UG \simeq U$. Then the equivalence admits a canonical refinement to an equivalence $G \simeq \id$.
\end{lemma}
\begin{proof}
We first observe that $G$ preserves limits and filtered colimits, since these are generated by the functor $U$. Together with the fact that $\SMC$ is presentable and the adjoint functor theorem, this shows that $G$ is right adjoint. 
Denote the left adjoint of $G$ by $F$. The equivalence $UG \simeq U$ implies that the diagram
\[
\xymatrix@-1pc{
&\Cat_\infty \ar[rd]^{\Fr}\ar[ld]_{\Fr} & \\ 
\SMC \ar[rr]^F & & \SMC
}
\]
commutes, where $\Fr$ is the free symmetric monoidal category functor. Now we use that the functor $\Fr$ exhibits $\SMC$ as the free prsesentable, pre-additive category on $\mathrm{Cat}_\infty$~\cite[Theorem 4.6.]{GGN}. 
Since $F$ is left adjoint this implies that it has to be canonically equivalent to the identity. Thus also the right adjoint $G$ is canonically equivalent to the identity.
\end{proof}
The proof of Proposition~\ref{prop:monoidal-equivalence} in fact provides the following stronger statement:
\begin{proposition}\label{prop_symmon}
For every symmetric monoidal functor $f\colon \cD \to \cD'$ between small $\infty$-categories there exists a symmetric monoidal, left Quillen functor between model categories $F\colon \cM \to \cM'$ such that 
$f_!\colon \cP(\cD) \to \cP(\cD')$ is symmetric monoidally equivalent to the underlying functor of $F$.\qed
\end{proposition}
\begin{proof}
We use Proposition \ref{prop:SMC-as-localization} to represent $f$ by a map of $\cE$-algebras. Then we get the induced functor between model categories and our proof of Proposition~\ref{prop:monoidal-equivalence} shows that this  models the $\infty$-functor $f_!$. 
\end{proof}



\end{document}